\title{Group Structures on Families of Subsets of a Group}

\author{Mario G\'omez, Sergio R. L\'opez-Permouth , Fernando Mazariegos, \\ Alejandro Jos\'e Vargas De Le\'on and Rigoberto Zelada Cifuentes}
\documentclass{article}
\usepackage{graphicx}
\usepackage{amsthm}
\usepackage{amssymb}
\usepackage{enumitem}

\newtheorem*{define}{Definition}
\newtheorem*{rmk}{Remark}
\newtheorem{thm}{Theorem}

\newcommand{\sG}{\mathcal G}

\begin{document}
\newpage
\maketitle
\begin{abstract}
    A binary operation on any set induces a binary operation on its subsets. We explore families of subsets of a group that become a group under the induced operation and refer to such families as {\it power groups} of the given group.  Our results serve to characterize groups in terms of their power groups.  In particular, we consider when the only power groups of a group are the factor groups of its subgroups and when that is the case up to isomorphism. We prove that the former are precisely those groups for which every element has finite order and provide examples to illustrate that the latter is not always the case. In the process we consider several natural questions such as whether the identity element of the group must belong to the identity element of a power group or the inverse of an element in a power group must consist of the inverses of its elements.
\end{abstract}



\section {Introduction}

Traditionally, group theory has been built upon understanding the subgroups and the factor groups of a given group.  In other words, our understanding of groups has largely rested on understanding their subquotients (factors of subgroups or subgroups of factors; it's the same.)  We propose the power groups of a group as a larger collection of groups that might give us a better insight into the structure of the original group.  The Spanish philosopher Ortega y Gasset said {\it I am I plus my circumstance} \cite{Ph}; inspired by that thought, we have been referring to the family of power groups of a group $G$ as the {\it circumstance} of $G$.  This paper is a first step as we  propose to study the way the structure of a group is determined by its circumstance. 

Reading this paper requires very little background; an undergraduate Modern Algebra course using a textbook like \cite {G} will suffice.

A binary operation on any set $S$ induces a binary operation on its subsets. Let us refer to the operation on the original set as {\it product}
and the one on its subsets as {\it subset product} Then, for subsets $A, B \subseteq S$,  the subset product
$AB$ equals $ \{ab ~|~ a \in A, b \in B\} $. Clearly not all families of subsets of a group $G$ will become groups under the induced subset product.
In fact, $\emptyset$ can only belong to one such family, namely $\{ \emptyset \}$. For convenience, we consider therefore only families of 
non-empty subsets in what follows. If $\sG \subseteq \mathcal P(G) - \{ \emptyset\} $ is a group with the above defined product, then $\sG$ is a \emph{power group} of the group $G$.

\section {Groups whose only power groups are factor groups}

The construction of factor groups, a standard topic in a first course in group theory for undergraduates, views the elements of such groups 
in one of two ways which are equivalent when the identity $N$ of the obtained structure is a normal subgroup.  For a group $G$ and a normal subgroup $N$ of $G$,
one can think of the operation on $\frac {G}{N}$ as being the multiplication of the subsets $aN$ and $bN$ with the binary operation induced on subsets
or think of the identity $(aN)(bN)=(ab)N$ as being a definition (which can be shown to be well-defined (independent from the choice of representatives) under the given hypotheses.  We will see that the two approaches are not necessarily equivalent when the hypothesis for the identity for the power group is eased.  Our considerations will give rise to two separate notions which we will name {\it subquotients} and {\it groups of cosets}. Subquotients are the subject of this section while groups of cosets will make their appearance in the next one.

Since the factor group $\frac {G} {N}$ of a group $G$ modulo a normal subgroup $N$  may be viewed as the group structure induced by the subset product on
the family of left (equivalently right) cosets of $N$ in $G$, $\frac {G} {N}$  is a natural example of a power group of the group $G$.  More generally, give any subgroup $H < G$ and a normal subgroup $N$ of $H$, the factor group $\frac {H} {N}$ is also a power group of $G$.  To facilitate the conversation, we give these examples a name.

\begin{define}
    A power group $\sG$ of $G$ is a subquotient group if there exists a subgroup $H$ of $G$, and a normal subgroup $N$ of $H$, such that $\sG$ consists precisely of the cosets of $N$ in $H$ (i.e.  $\sG = \frac {H} {N}$.)
\end{define}
The next theorem shows that subquotients are characterized by many natural properties.
\begin{thm}
    The following are equivalent for a power group $\sG$ of $G$:
    \begin{enumerate}[label=(\alph*)]
        \item $\sG$ is a subquotient group of G.
        \item $E$, the identity element of $\sG$, is a subgroup of $G$.
        \item For every $A \in \sG$, if $x \in A$ then $x^{-1} \in A^{-1}$
        \item $H = \bigcup_{A \in \sG} A $ is a subgroup of $G$, and $\sG$ is a partition of $H$.  
    \end{enumerate}
\end{thm}

\begin{proof}
    Clearly $(a) \rightarrow (b)$, because then $\sG=\frac {H}{N}$ for some $H,N$ and the identity of $\sG$ is $N$. 
    
    To prove $(b) \rightarrow (c)$, take $x \in A$ and, by virtue of the fact that $\emptyset \notin \sG$, choose $y \in A^{-1}$.  It follows that $xy \in AA^{-1} = E$;  since $E$ is a group, there is $z \in E$ such that $(xy)z = e$. Finally, since $yz \in A^{-1}E = A^{-1}$, then we have $x^{-1} = yz \in A^{-1}$, proving (c).  
    
     To prove $(c) \rightarrow (d)$, let $H = \bigcup_{A \in \sG} A $; it is easy to see that $H$ is always closed under products and, under the assumption of (c), also under inverses.  Therefore, $H$ is a subgroup as claimed. To see that $\sG$ is a partition, suppose$A, B \in \sG$ share an element $x \in A \cap B$. Then,
under our assumption, $A^{-1}$ and $B^{-1}$ also share an element as $x^{-1} \in A^{-1}\cap B^{-1}$.  But then $e \in A^{-1}B$ and $e \in B^{-1}A$. This means that $A \subseteq AA^{-1}B = B$ and $B \subseteq BB^{-1}A = A$, so $A = B$ and therefore $\sG$ is a partition. 

    To prove $(d) \rightarrow (a)$, we aim to show that $\sG= \frac {H} {E}$ where $E$ is the identity element of $\sG$. 
Some notation will be convenient: for $x \in H$, let us use $A_x$ to denote the (unique) element of $\sG$ containing $x$. 
Our first aim is to show that $E= A_e$. As $e \in A_e^2$, considering that $\sG$ is a partition, $A_e^2 = A_e$. Cancellation in the group $\sG$ yields that $A_e = E$, as claimed; in particular $e \in E$. Next, we show that the elements of $\sG$ are the cosets of $E$, which then concludes the proof.  Taking $A_x \in \sG$ , the fact that $H$ is a  subgroup of $G$ yields that $A_{x^{-1}} \in \sG$.Under our assumption, noting that $e \in A_xA_{x^{-1}}$ implies $A_x^{-1} = A_{x^{-1}}$. 
For an arbitrary $x \in H$,
consider the coset $Ex \subseteq EA_x = A_x$. We want to get the reverse inclusion to prove that $ A_x= Ex$; for that purpose, pick $y \in A_x $. Now, $A_x = A_y$ and $A_{x^{-1}} = A_{y^{-1}}$.  Considering $A_yA_{x^{-1}} = A_yA_{y^{-1}} = E$ one concludes that $yx^{-1} \in E$ and, from this, $y \in Ex$.  
     
\end{proof}

Armed with an understanding of subquotients, we aim now to show that not all power groups are in general subquotients; in fact, the following theorem 
provides a characterization of those groups for which the only power groups are the subquotients.

\begin{thm}
    Let $G$ be a group, then the only power groups of G are the subquotients of $G$  if and only if for all $a \in G$ $ord(a) < \infty$.
\end{thm}

\begin{proof}
($\leftarrow$) By virtue of Theorem 1, it suffices to show that  $E$, the identity element of $\sG$, is a subgroup.  That is indeed the case since $EE=E$
yields closure and, since by assumption every element of $G$ is of finite order,$E$ is also closed under inverses and is consequently a subgroup of $G$.

($\rightarrow$) Suppose there exists an $a\in G$ such that $ord(a)=\infty$ and consider the set $E=\left\{ a^{n} | n\in \mathbb {N}\right\} \subseteq \mathcal P(G)$, clearly $EE=E$ and therefore $\sG=\left\{ E \right\}$ is a power group of $G$. Notice that $E$ is the identity element of $\sG$ and is not a group, thus $\sG$ is not a subquotient of $G$.  
\end{proof}

\section {What lies ahead}

Central to the understanding of power groups is the analysis of their identity elements.  While in the case of subquotients, the identity element must be a subgroup, in general all one needs is a subset $E \subset G$ such that $E^2=E$.  Assume, for example that one has such a subset $E$ of a group $G$ and say that $E$, in addition, satisfies a {\it normalcy} property that for all $a \in G$, $aE=Ea$. Examples of appropriate choices of $E$ include, for the additive group of integers $G= \mathbb Z$, the set $E$ of non-negative integers and, for the additive group of rational numbers $G= \mathbb Q$, the set $E= \mathbb Q^+$ of positive rationals.  

Following the steps of the usual proof to show that, for a normal subgroup $N$ of $G$, $\frac {G}{N}$ is a group, one obtains the same result here with weaker hypotheses.  In other words, it is straightforward to show that the group in the following definition is indeed a group:

\begin{define} For any group $G$, given $E \subset G$ such that $E^2=E$ and a subgroup $H < G$; assume, further, that for every $a\in H$, $aH=Ha$, then
the family $\{ aE|a \in H \}$  is a power group which will be called the group of cosets induced by $E$ and $H$ and will be denoted $\frac {H}{E}$.  It satisfies that, for any $a,b \in H$, $(aE)(bE) = (ab)E$.  We will refer to any such power group as a {\it group of cosets} of the group $G$. Notice that the definition does not require that $E$ be a subset of $H$.
\end{define}

\begin{rmk}\label{cosetsaresubquotients}  Every group of cosets is isomorphic to a subquotient.
\end{rmk}
\begin {proof}
Define a map $\varphi: H \rightarrow \frac {H}{E}$ via $\varphi(a)=aE$ and verify it is a group epimorphism. Notice that subquotients are always groups of cosets but, in general, groups of cosets need not even be partitions.  

\end {proof}

An interesting feature of this family of power groups  is that not only is the identity element not necessarily a subgroup but actually it 
does not even need contain the identity element of the group (as seen in the example $E= \mathbb Q^+$ mentioned above.)
Motivated by the realization that the traditional construction yields more power groups, it is natural to consider a modification of the original question from this paper and ask for which groups it is the case that all power groups are groups of cosets.  In fact, in light of Remark \ref{cosetsaresubquotients}, another natural question is to characterize those groups $G$ all power groups of $G$ are isomorphic to subquotients of $G$.  As we finish this project, we do not yet know the answers to either one of these questions; however, we can show that $\mathbb Z$ is an example of a group whose only power groups are its groups of cosets and that $\mathbb Q$ is an example of the opposite situation; in fact $\mathbb Q$ is a group having power groups that are not isomorphic to subquotients.  

\begin{thm}    Every power group of $\mathbb Z$ is a group of cosets.
\end{thm}

\begin{proof}
One must determine first that the only possible non-empty choices for an identity element $E$ of a powergroup $\mathcal G$ of $\mathbb Z$. One easily gets that, outside subgroups of $\mathbb Z$, $E$ must be a subset of $\mathbb Z$ which contains $0$ and is such that either all its other elements are positive or all of them are negative.  Focus, without losing generality, on the case when the elements of $E$ are non-negative; the other case is handled similarly.  One can easily show that every element $A \mathcal G$ has a smallest element, say $a \in A$ and $A=aE$. Incidentally, $a$ may also be recognized by the property that it is the one element of $A$ such that $-a \in -A$.  Denote $H$ the set of smallest elements in elements of $\mathcal G$; it follows that $\mathcal G = \frac {H}{E}$.

\end{proof}
\begin{thm}
    The additive group of rational numbers $\mathbb Q$ is an example of a group having a power group which is not a group of cosets.
\end{thm}

\begin{proof}
Let $\mathcal G = \{ (r, \infty) \cap \mathbb Q| r \in \mathbb R \}$.  Then $\mathcal G \cong \mathbb R$, in particular $|\mathcal G| > |\mathbb Q|$ hence $\mathcal G$ is not isomorphic to a subquotient.
\end{proof}

In general, it seems interesting to be able to determine, given two groups $G_1$ and $G_2$, when $G_1$ has a power group isomorphic to $G_2$.  
 For convenience, let us use the expression {\it $G_2$ underlies $G_1$} when $G_1$ has a power group isomorphic to $G_2$. The last example shows an instance of two non-isomorphic groups that underlie one another. What properties are shared by groups that underlie one another? Another interesting question that we suggest here is whether {\it underlying} is a transitive relation, namely whether, for three groups $G_i$  $(i = 1 , 2, 3)$, if $G_3$ underlies $G_2$ and $G_2$ underlies $G_1$, can one conclude that $G_3$ underlies $G_1$?  Notice that if all elements of $G_1$ have finite order then the same would be true of $G_2$ and $G_3$ and (since  a subquotient of a subquotient is a subquotient) concluding that the result holds would be a straightforward application of Theorem 1; settling the general case seems to be more difficult.

\end{document}